\documentclass[10pt]{article}
\usepackage[utf8]{inputenc}
\usepackage{amssymb}
\usepackage{amsmath}
\usepackage{amsthm}
\usepackage{url}
\usepackage{mathrsfs}
\usepackage{epigraph}

\newtheorem{prop}{Proposition}
\newtheorem{lem}{Lemma}

\newtheorem{theor}{Theorem}
\newtheorem{cor}{Corolary}
\newtheorem{rem}{Remark}

\newcommand{\N}{\mathbb{N}}
 
\newcommand{\Z}{\mathbb{Z}}

\DeclareMathOperator{\rank}{rk}

\begin{document}
%\nocite{LMSLN310}
%\nocite{OrWe87}

\author{Andrei Alpeev  \footnote{Euler Mathematical Institute at St. Petersburg State University, alpeevandrey@gmail.com}}
\title{Examples of measures with trivial left and non-trivial right random walk tail boundary}
%\footnote{This research is supported by the Chebyshev Laboratory  (Department of Mathematics and Mechanics, St. Petersburg State University)  under RF Government grant 11.G34.31.0026, by RSF grant №14-11-00581 and by JSC "Gazprom Neft"}

%\newpage
%\newpage
%\newpage
\maketitle
\begin{abstract}
In early 80's Vadim Kaimanovich presented a construction of a non-degenerate measure, on the standard lamplighter group, that has a trivial left and non-trivial right random walk tail boundary. We show that examples of such kind are possible precisely for amenable groups that have non-trivial factors with ICC property.
\end{abstract}

\epigraph{To the memory of my dear teacher Anatoly Moiseevich Vershik.}

\section{Introduction}
Let $G$ be a countable group and $\nu$ be a probability measure on $G$.  A measure on $G$ is called {\em non-degenerate} if its support generates $G$ as a semigroup. 
The $\nu$-random walk on $G$ is defined in the following way. First let $(X_i)_{i=1}^{\infty}$ be the i.i.d. process with distribution $\nu$. We set $Z_i = X_1 \cdot \ldots \cdot X_i$. Process $(Z_i)$ is called the right $\nu$-random walk on $G$. Similarly, we can define the left random walk by setting $Z_i' = X_i \cdot \ldots \cdot X_1$.  
By default, {\em random walk} will mean right random walk. We will restrict ourselves to non-degenerate measures on groups.
If $\nu$ is a measure on a countable group $G$, we may define an opposite measure $\nu^{-1}$ by $\nu^{-1}(g) = \nu(g^{-1})$.
It is trivial to see that instead of left random walks, we may consider right random walks with opposite measures. 
The {\em tail boundary} $\partial(G, \nu)$ or the {\em tail subalgebra} of random walk $(Z_i)$ is defined as the intersection $\bigcap_j \sigma(Z_j, Z_{j+1}, \ldots)$, where $\sigma(Z_j, Z_{j+1}, \ldots)$ denotes the minimal $\sigma$-algebra under which all variables $Z_j, Z_{j+1}, \ldots$ are measurable. We will habitually write ``subalgebra'' instead of ``$\sigma$-subalgebra''.
%As customary in ergodic theory, we will identify it with the corresponding factor-space $\partial(G, \nu)$.
Pair $(G,\nu)$ (or, abusing notation, measure $\nu$ itself) is called Liouville if the tail boundary of $\nu$-random walk on $G$ is trivial. 
One of the fundamental questions of asymptotic theory of random walks is whether a measure on a group is Liouville. Another notion of boundary is that of the {\em Poisson boundary}, it is defined as the invariant-set subalgebra  of the process $(Z_i)_{i \in \N}$ under the time-shift action; in the setting of the random walk on group with non-degenerate measure, the Poisson-Furstenberg boundary coincides with the tail boundary (see \cite{KaVe83}, \cite{Ka92}), so we will use these notions interchangeably.
Due to the Kaimanovich-Vershik entropy criterion for boundary triviality \cite{KaVe83}, we have that if a measure $\nu$ on $G$ has finite Shannon entropy (defined by $H(\nu) = -\sum_{g \in G}\nu(g) \log\nu(g)$, assuming $0 \log 0 = 0$), then left and right $\nu$-random walks have trivial tail boundaries simultaneously. Surprisingly, this is not the case if the finite entropy assumption is waived: in \cite{Ka83} Kaimanovich constructed an example of a measure on the standard lamplighter group $\Z/2\Z \wr \Z$ such that the left random walk has trivial tail boundary, while the right random walk has non-trivial.
The purpose of the present note is to explore which countable groups admit examples akin to that of Kaimanovich. Our main result is the following:

\begin{theor}
Let $G$ be a countable group. There is a non-degenerate probability measure $\nu$ on $G$ with trivial left and non-trivial right random walk tail boundaries iff $G$ is amenable and has a non-trivial ICC factor-group. Moreover, the action of maximal ICC factor on the boundary is essentially free.
\end{theor}
In the process, we also characterize rather explicitly the boundary for our example in the manner of \cite{ErKa19}.

We remind that a group is called an ICC (short for infinite conjugacy classes) if conjugacy class of each nontrivial element of the group is non-trivial. Note that a finitely-generated group lacks a non-trivial ICC factor exactly when it is virtually-nilpotent (=has polynomial growth, due to the famous Gromov theorem), see \cite{DuM56}, \cite{M56}.  

%We note that using more subtle techniques from \cite{ErKa19}, one can prove that the boundary is not only non-trivial, but the action of any ICC factor-group on the corresponding factor-boundary could be made to be essentially free. In this note we only show that the boundary is non-trivial.

It is well known that amenable groups and only them admit non-degenerate Liouville measures, see Theorems 4.2 and 4.3 from \cite{KaVe83}). It is also well known that all measures on groups without ICC factors are Liouville, see \cite{Ja}, a self-contained proof could be also found in the second preprint version of \cite{Feta19}. 
Thus examples of Kaimanovich type are possible only for amenable groups with non-trivial ICC factors. In the sequel we will show  that for every such group there is a measure of full support with non-trivial right and trivial left random walk boundary. Our construction is based on that of the breakthrough paper \cite{Feta19} of Frish, Hartman, Tamuz and Vahidi Ferdowsi, where a non-Liouville measure was constructed for every group with an ICC factor, combined with the classic construction of a Liouville measure for every amenable group by Kaimanovich and Vershik \cite{KaVe83} and Rosenblatt \cite{Ro81}, although in the proof of non-triviality of boundary we employ the approach similar to that of Ershler and Kaimanovich \cite{ErKa19}.

An interesting consequence of the present result is an explicit construction of a mean on a countable group that is right-invariant but not left-invariant. 
It was shown by Paterson \cite{Pat79}, that for countable groups such examples are possible exactly for FC-groups (i.e. groups where each element has finite conjugacy class).
For a countable group $G$ consider the space $l^{\infty}(G)$, we endow it with the R-representation of $G$ by 
$$( \rho^g f)(h) = f(hg), \text{ for } g,h \in G \text{ and }f \in l^{\infty}(G),$$
and with the L-representation of $G$ by
$$(\lambda^g f)(h) = f(gh).$$

A mean on $l^{\infty}(G)$ is a positive functional of norm $1$. A mean $\mathfrak{m}$ is left-invariant if $\mathfrak{m}(f) = \mathfrak{m}(\lambda^g (f))$ for all $g \in G$ and $f \in l^{\infty}(G)$, and it is right-invariant if $\mathfrak{m}(f) = \mathfrak{m}(\rho^g (f))$, for all $g \in G$ and $f \in l^{\infty}(G)$.
The following observation is made in the work of Kaimanovich and Vershik \cite[Section 6.5]{KaVe83}, see also \cite[Theorem 1.4]{Ka83}:
\begin{prop}
Let $\nu$ be a non-degenerate measure on a countable group $G$ such that the boundary $\partial(G, \nu)$ is non-trivial and $\partial(G, \nu^{-1})$ is trivial. Let $\mathfrak{u}$ be a principal ultrafilter, The functional $\mathfrak{m}$ defined by
$$ \mathfrak{m}(f) = \lim_{i \to \mathfrak{u}} \sum_{g \in G} \nu^{*i}(g) f(g) \text{, for } f \in l^{\infty}(G),$$ is a right-invariant but not left-invariant mean. 
\end{prop}
%Since our construction of the measure is rather explicit we get the following.
\begin{cor}
For every amenable group with a non-trivial ICC factor, there is an ``explicit'' right-invariant mean that is not left-invariant.
\end{cor}
Of course, we rely on the Axiom of Choice as we use an ultrafiler, but this is a rather localized use.
\begin{proof}
It is trivial that the funcional defined is a mean. To prove that it is right-invariant, we note that for any $g \in G$ we have, by Proposition \ref{lem: KV invariance criterion}, that 
$$\lim_{i \to \infty} \lVert g * (\nu^{-1})^{*i} -  (\nu^{-1})^{*i} \rVert = 0,$$
and so
$$\lim_{i \to \infty} \lVert \nu^{*i} * g -  \nu^{*i} \rVert = 0.$$
Now we get 
\begin{gather*}\mathfrak{m}(\rho^g f) = \lim_{i \to \mathfrak{u}} \sum_{h \in G} (\rho^g f)(h) \nu^{*i}(h) = \lim_{i \to \mathfrak{u}} \sum_{h \in G} f(hg), \nu^{*i}(h) \\=  \lim_{i \to \mathfrak{u}} \sum_{h \in G} f(h) (\nu^{*i} * g)(h) = \lim_{i \to \mathfrak{u}} \sum_{h \in G} f(h) (\nu^{*i})(h) = \mathfrak{m}(f).
\end{gather*}

To prove that $\mathfrak{m}$ is not left-invariant, we observe that, since the boundary $\partial(G,\nu)$ is non-trivial, there is a non-constant bounded $\nu$-harmonic function $f$, that is 
$$f(g) = \sum_{h \in G} f(gh) \nu(h)\text{, for all } g\in G.$$
Now we compute:
$$\mathfrak{m}(\lambda^g(f)) = \lim_{i \to \mathfrak{u}} f(gh) \nu^{*i}(h) = f(g),$$
and since $f$ is non-constant, we trivially get that $\mathfrak{m}$ is not left-invariant.

\end{proof}

{\bf Acknowledgements.} I'm thankful to Vadim Kaimanovich for helpful discussions.
I thank Anna Erschler for suggestions on the manuscript.
Prof A. M. Vershik constantly pushed me to study asymptotic behaviour of Markov chains, for which I'm grateful. Sadly, he passed away on February 14th, 2024.
This research project was started during my postdoc in the Einstein Institute of Mathematics at the Hebrew University of Jerusalem, supported by the Israel Science Foundation grant 1702/17, and finished in the Euler Mathematical Institute at Saint-Petersburg State University. The work is supported by Ministry of Science and Higher Education of the Russian Federation, agreement № 075–15-2019-1619.

\section{The maximal ICC factor}
Let $G$ be a group. Let $\nu$ be a non-degenerate probability measure on $G$, we will denote $\partial(G,\nu)$ the Poisson-Furstenberg(= tail in this case) boundary of the $\nu$-random walk.

A group is said to have the {\em infinite conjugacy class property (ICC)} if every its non-trivial element has infinite conjugacy class.
See e.g. \cite[Section 5B]{ErKa19} for the following:
\begin{lem}
Every group has a maximal ICC factor-group.
\end{lem}
Let $\varphi : G \to \Gamma$ be an epimorphism of groups. %Let $\nu$ be a non-degenerate measure on the group. 
It is not hard to see that we have an induced factor-map of tail (and hence, Poisson-Furstenberg) boundaries $\varphi_{\partial}: \partial(G, \nu) \to \partial(\Gamma, \varphi_*(\nu))$.
%We need the following theorem:

The next theorem follows easily from \cite[Proposition 5.1, Proposition 5.11]{ErKa19} and the fact that the kernel of the canonical map onto the maximal ICC factor is the so-called hyper-FC-center.
\begin{theor}
Assume $G$ is a countable group, $\nu$ --- a non-degenerate measure on $G$ and $\Gamma$ is the maximal ICC factor of $G$ together with the canonical epimorphism $\varphi: G \to \Gamma$. The natural map $\varphi_{\partial}: \partial(G, \nu) \to \partial(\Gamma, \varphi_*(\nu))$ is an isomorphism.
\end{theor}

In this note our goal it to construct for any countable amenable group $G$ with a non-trivial ICC factor, a probability measure of full support $\nu$ such that $\partial(G ,\nu)$ is non-trivial and $\partial(G, \nu^{-1})$ is trivial.
Note that it is enough to do this for the maximal ICC factor, since we can take any lift-measure $\nu$ (i.e. $\varphi_*(\nu) = \nu'$) of full support on $G$, and in view of the theorem above, $\partial(G, \nu)$ is naturally isomorphic to $\partial(\Gamma, \nu')$, and $\partial(G, \nu^{-1})$ is naturally isomorphic to $\partial(\Gamma, \nu'^{-1})$.
In the sequel we will assume that the group considered is a non-trivial countable ICC amenable group.

\section{A process with heavy tail}
Let $K$ be an integer-valued random variable such that $P(K = k) = (1/c) k^{-5/4}$, for $k \in \N$. Consider an i.i.d. process $(K_i)_{i \in \N}$ (each $K_i$ has the same distribution as $K$). A number $i \in \N$ is a record-time if $K_i \geq K_j$ for all $j < i$, and the value $(K_i)$ is a record-value; we will call pair $(i,K_i)$ a record, and usually denote it, abusing notation a bit, as $K_i$. A record $K_i$ is simple if $K_i \neq K_j$ for all record-times $j$ such that $j \neq i$.

The following lemma could be found in \cite[Lemma 2.6]{Feta19} and \cite[Sections 2.B and 2.C]{ErKa19}.
\begin{lem}
For almost every realization of the random process $(K_i)$, there is $i_0$ such that 
\begin{enumerate}
\item for all $i \geq i_0$ we have $\max\{ K_1, \ldots K_i\} > i$;
\item all record-times starting from  $i_0$ are simple.
\end{enumerate}
\end{lem} 

\begin{rem}
Instead of this particular process we can use an arbitrary process with simple records as described in \cite[Sections 2.B and 2.C]{ErKa19}, with minor adjustments of parameters in the construction that follows.
\end{rem}

We have a random variable $K$, let us construct coupled random variable $Y$. If $K = k_0$, we set $Y = ``red''$ with probability $2 ^ { -k_0}$ and $Y = ``blue''$ with probability $1 - 2^{-k_0}$. Now consider the process $(K_i, Y_i)_{i=1}^{\infty}$ such that the sequence of independent pairs $(K_i, Y_i)$ distributed as we described.

Consider a trajectory of the random process $(K_i,Y_i)_{i \in \N}$.  We will say that this trajectory {\em stabilizes} if there is $i_0$ such that 
\begin{enumerate}
\item for all $i \geq i_0$ we have $\max\{ K_1, \ldots K_i\} > i$;
\item all record-times $i$ starting from  $i_0$ are simple and $Y_i = ``blue''$ for these record-times.
\end{enumerate}
We will call the smallest such $i_0$ (if it exists) the {\em stabilization time}.
Now it is easy to extend the previous lemma in the following way using the Borel-Cantelli lemma:
\begin{lem}\label{lem: basic stabilization}
Almost every realization of the random process $(K_i,Y_i)_{i \in \N}$ stabilizes.
\end{lem}

\section{Construction}

Let $G$ be a group, and $A$ be a subset of $G$. We will say that a finite subset $F$ of $G$ is $(A,\delta)$-invariant if $\lvert aF \setminus F\rvert < \delta \lvert F\rvert$ for all $a \in A$. Naturally, in an amenable group there is an $(A,\delta)$-invariant set for any choice of finite $A \subset G$ and $\delta > 0$.

Let $\Gamma$ be a group. Let $A$ be a finite subset of $\Gamma$. We will say that an element $b \in \Gamma$ is an $A$-lock if for any $a'_1, a'_2, a''_1, a''_2$ from $A$, equality $a'_1ba'_2 = a''_1 b a''_2$ implies $a'_1 = a''_1$ and $a'_2 = a''_2$, and sets $A$ and $AbA$ are disjoint.

The proof of the following for amenable groups could be found in \cite[Proposition 2.5]{Feta19} and in the general case in \cite[Poposition 4.25]{ErKa19}.
\begin{lem}
If $\Gamma$ is a non-trivial ICC group, then for every finite subset $A$ of $\Gamma$ there is an $A$-lock. 
\end{lem}

We will construct the measure $\nu$ on $G$ for the main theorem as a distribution of a certain $G$-valued random variable $X$ coupled with $(K,Y)$.

Let $(c_i)$ be any sequence enumerating all the elements of $G$. We will construct the variable in an iterative manner, together with finite subsets $A_i, F_i, D_i \subset G$, and elements $b_i \in G$ for each $i \in \N$.

\begin{enumerate}
\item Let $A_1 = \{e\}$. 
\item For each $i \geq1$ we choose $F_i$ to be $((A_i \cup \{ c_i\} \cup \{c_i^{-1}\})^{i+1}, 1/i)$ - invariant. 
\item We denote $D_i = F_i^{-1} \cup F_i \cup A_i \cup \{c_i\} \cup\{c_i^{-1}\}$, for $i \in \N$.
\item For each $i \geq 1$ we choose $b_i$ to be a $D_i^{10i + 10}$-lock. 
\item For each $i \in\N$ we set $A_{i+1} = D_{i}\cup b_i F_i^{-1} \cup F_i b_i^{-1}$.
\end{enumerate}

So we start with step 1, proceed to step $5$ and then carry on repeating steps 2--5.

We are ready to construct a random variable $X$ that is coupled to $(K, Y)$. Assume $K = i$. If $Y = ``red"$, we set $X = c_i$. Otherwise let $X$ be uniformly distributed in $b_i F_i^{-1}$. 

So let $\nu$ be the distribution of $X$. It is trivial that the support of $\nu$ is $G$. It would be convenient at times for us to look at the i.i.d. process $(X_1, X_2, \ldots)$ as coupled with the i.i.d. process $((K_1, Y_1, X_1),(K_2, Y_2, X_2), \ldots)$.

The following proposition appears as a part of Theorem 4.2 from \cite{KaVe83}:
\begin{prop}\label{lem: KV invariance criterion}
Let $\nu$ be a non-degenerate measure on a countable group $G$. The Poisson-Furstenberg boundary of $\nu$-random walk on $G$ is trivial iff for every $g \in G$ we have $\lVert g * \nu^{*n} - \nu^{*n} \rVert \to 0$.
\end{prop}

\begin{lem}
$\nu^{-1}$ - random walk on $G$ has trivial Poisson-Furstenberg boundary.
\end{lem}

\begin{proof}
Let $g$ be fixed. Assuming that $n$ is big enough, the sequence $K_1, \ldots, K_n$ with probability close to $1$ has unique maximal value $m = K_i > n$, and the corresponding $Y_i = ``blue''$; this is a trivial consequence of Lemma \ref{lem: basic stabilization}.
So we have that $(\nu^{-1})^{*n}$ could be decomposed as 
\[
(\nu^{-1})^{*n} = \sum_{q',q'',m} p_{q',q'',m} \cdot q' * \lambda_{F_m} b_n^{-1} q'' + \eta_n, 
\]
where $q', q'' \in A_m^n$, $m>n$, $p_{q',q'',m} \geq 0$, $\lambda_{F_m}$ is the uniform measure on $F_m$, and $\lVert \eta_n \rVert \to 0$ as $n \to \infty$. From this we readily conclude that $\lVert(\nu^{-1})^{*n}- g* (\nu^{-1})^{*n}\rVert \leq 4/n + 4\lVert \eta_n\rVert$, as soon as $g \in A_n$, so the assumption of Proposition \ref{lem: KV invariance criterion} is fulfilled, since $A_i$ is a growing sequence of finite subsets whose union is whole $G$.
\end{proof}

Now we will show that the tail boundary for the $\nu$-random walk is nontrivial. For this we will construct a tail-measurable function $\tau$ and show that its image is not essentially a one-point set. We will also demonstrate that the natural $G$-action on the boundary of the $\nu$-random walk on $G$ is essentially free. Moreover, we will show that function $\tau$ actually defines the boundary (its image could be identified with the factor-space of the Poisson-Furstenberg boundary).

Denote $W_n = A_n^{n+1} b_n F_n^{-1} A_n^n$ and $W'_n = A_n^n b_n F_n^{-1} A_n^n$ for all $n \in \N$. Also set $W_0 = G \setminus \bigcup_{i \in \N} W_i$. Trivially, $W'_n \subset W_n$ for all $n$ (since $e \in A_n$).

For $\gamma \in G$ we define $\rank(\gamma)$ to be the unique $i$ such that $\gamma \in W_i$ or $0$ if there are no such $i$.

Let $p: \bigcup_n W_n \to G$ be a function defined by the formula $p(q' b_n f^{-1} q'') = q'$, where $q' \in A_n^{n+1}$, $q'' \in A_n^n$, $f \in F_n$. Note that $p$ is defined properly since by construction $b_n$ is an $D_n^{10n+10}$-lock, and $D_n = F_n^{-1} \cup F_n \cup A_n \cup \{c_n\} \cup\{c_n^{-1}\}$.
%and for all $\gamma \in G$ there is $n_0$ such that $\gamma W'_n \subset W_n$ for all $n > n_0$. Observe that $\rank(p(\gamma)) < \rank(\gamma)$ if $\rank(\gamma) > 0$ and $\rank(p(\gamma)) = 0$ if $\rank(\gamma) = 0$.
%Observe also that for any $\gamma \in W'_n$ and $\gamma_1 \in \varphi(A_i)$ for $i \geq n$ we have $p(\gamma_1\gamma) = \gamma' p(\gamma)$.

\begin{lem}
We have the following:
\begin{enumerate}
\item $W_i \cap W_j = \varnothing$ for $i \neq j$;
\item $W'_i \subset W_i$, for $i > 0$;
\item for any $\gamma \in W'_i$ and $\gamma_1 \in A_n$ for $i \geq n$ we have $p(\gamma_1\gamma) = \gamma' p(\gamma)$;
\item if $\rank(\gamma) > 0$ then $\rank(p(\gamma)) < \rank(\gamma)$;
%\item if $\gamma \in W'_n$ and $\gamma' \in \varphi(A_i)$ for any $i \geq n$ then $p(\gamma_1 \gamma) = \gamma_1 p(\gamma)$.
\end{enumerate}
\end{lem}
\begin{proof}
For (1) we assume that $i < j$. Remind that $b_j$ is a $D_j^{10j + 10}$-lock, and $W_i \subset D_i^{10i + 10} \subset D_j^{10j + 10}$, so the required follows by the definition of a lock.

For (2) we just use that $e \in A_i$ for all $i$.

For (3) we note that $\gamma = q' b_i f^{-1} q''$ for some $q', q'' \in A_i^i$ and $f \in F_i$, so $\gamma_1 \gamma =  \gamma_1 q' b_i f^{-1} q''$. We get that $p(\gamma) = q'$ and $p(\gamma_1 \gamma) = \gamma_1 q'$.

For (4) we note that if $\rank(\gamma) = i > 0$ then $\gamma \in W_i$, so $p(\gamma) \in A_i^{i+1}$, but $A_i^{i+1}$ is disjoint from $W_j$ for all $j \geq i$, by definition of the lock $b_j$.

\end{proof}

For every $\gamma \in G$ let us define $t(\gamma)$ to be the sequence $(s_i)_{i = 1}^m$ such that $m$ is the smallest with $p^m(\gamma) \in W_0$ and $s_i = p^{m-i+1}$, for $i = 1, \ldots, m$ (the sequence is empty if $\gamma \in W_0$). In other words, $t(\gamma) = (p^m(\gamma),\ldots, p^2(\gamma), p(\gamma))$.

For two sequences $\sigma_1, \sigma_2$ we will write $\sigma_1 \subset \sigma_2$ if $\sigma_1$ is an initial segment of $\sigma_2$ (in particular, we might have $\sigma_1 = \sigma_2$).
If $\sigma_1 \subset \sigma_2 \subset \ldots$ is a growing sequence of sequences, we denote $\bigcup_i \sigma_i$ the minimal sequence (possibly infinite) $\sigma'$ such that $\sigma_i \subset \sigma'$ for all $i$.
For a finite sequence $\sigma$ and an element $s$ we will denote $\sigma \sqcup s$ the sequence $\sigma$ with $s$ attached at the end.  

Remind that we consider the random walk defined by $Z_i = X_1 \cdot \ldots \cdot X_i$.
\begin{lem}
For a.e. trajectory $(z_i)_{i=0}^{\infty}$ of the $\nu$-random walk, there is an $N$ such that:
\begin{enumerate}
\item for every $i > N$ there is $j$ such that $z_i \in W'_j$;
\item for every $i > N$ either $t(z_{i+1}) = t(z_i)$ or $t(z_{i+1}) = t(z_i) \sqcup z_i$;
\item for every $i,j > N$ with $i < j$ we have $t(z_i) \subset t(z_j)$;
\item the sequence $\lvert t(z_i) \rvert$ is unbounded.
\end{enumerate}
\end{lem}

\begin{proof}
It will be convenient for us to consider the trajectory $(k_i, x_i, y_i)_{i \in \N}$ of the extended triplet-process $(K_i, X_i, Y_i)_{i \in \N}$, and the corresponding trajectory of the random walk $z_i = x_1 \cdot \ldots \cdot x_i$.
Assume that trajectory $(k_i, y_i)_{i \in \N}$ stabilizes (this happens, by Lemma \ref{lem: basic stabilization}, with probability $1$), let $N$ be the first record-time that is bigger than the stabilization time. We note that by definition of the stabilization time, if $i > N$ and $1 \leq j \leq i$ is the record on the segment $1, \ldots i$, i.e.  $k_j = \max\lbrace k_1, \ldots, k_i\rbrace$, then $z_i = x_1 \cdot x_{j-1} \cdot b_{k_j} f \cdot x_{j+1} \cdot \ldots \cdot x_i$, where $x_1,\ldots, x_{j-1}, x_{j+1}, x_i \in A_{k_j}$ and $f \in F_{k_j}^{-1}$ (since the record $k_j$ is simple and $y_j =$ ``blue''). We get trivially that $z_i \in W'_{k_j}$ (since $k_j > i$, by definition of the stabilization time), this gives us (1). 

Now consider $k_{i+1}$.  One possibility is $k_{i+1} < k_{j}$ (the record is not beaten) and then we get that  $z_{i+1} = x_1 \cdot x_{j-1} \cdot b_{k_j} f \cdot x_{j+1} \cdot \ldots \cdot x_i \cdot x_{i+1}$, where $x \in A_{k_j}$, so it is easy to check that $p(z_{i+1}) = p(z_i) = x_1 \cdot \ldots \cdot x_{j-1} = z_{j-1}$ (note that $i+1 < k_j$ since $k_j$ is a record on a segment $1,\ldots, i+1$ after the stabilization time), hence $t(z_{i+1}) = t(z_i)$. 

Another possibility is that $k_{i+1} > k_{j}$ (the option $k_{i+1} = k_{j}$ is excluded by the definition of the stabilization time: the record $k_j$ should be simple). In this case we get $z_{i+1} = x_1 \cdot \ldots \cdot x_i \cdot x_{i+1}$ where $x_{i+1} = b_{k_{i+1}} f'$ and $f \in F^{-1}_{k_{i+1}}$ (we remind that $y_{i+1}$=``blue'' by definition of the stabilization time). Note that $x_1, \ldots , x_i \in A_{k_{i+1}}$ $p(z_{i+1}) = x_1 \cdot \ldots \cdot x_i = z_i$, and in this case 
$$t(z_{i+1}) = t(p(z_{i+1})) \sqcup p(z_{i+1}) = t(z_i) \sqcup z_i.$$
Thus we get (2), and (3) trivially follows by induction.

For (4) is is enough to note that we get a new element in the sequence every time we get a new record, this trivially happens infinitely often after the stabilization time (since the record on segment $1, \ldots , i$ should be bigger than $i$ for all $i > N$).

\end{proof}

We define a function $\tau$ on a Borel set of full measure of trajectories $(z_i)$
by $\tau((z_i)_{i=1}^{\infty}) = \bigcup_{i > N} t(z_i)$. Note that for almost every $\zeta$, $\tau(\zeta)$ is an infinite sequence.

The observations of the next Lemma are almost trivial consequences of the previous one.
\begin{lem}\label{lem: tau properties}
Let $\tau(\zeta) = (s_i)_{i=1}^{\infty}$. For a.e. trajectory $\zeta= (z_i)_{i=0}^{\infty}$ of the random walk given by the distribution $\nu$,
\begin{enumerate}
\item $\tau$ is a tail-measurable sequence;
\item $s_i = p(s_{i+1})$ for all $i \in \N$;
\item we have $0 = \rank(s_1) < \rank(s_2) < \ldots$;
\item there is $N \in \N$ and a growing sequence $(q_i)$ such that $s_{i+N} = z_{q_i}$ for all $i > 0$;
\item there is $N \in \N$ such that for all $i>N$ we have $s_i \in W'_{\rank(s_i)}$.
\end{enumerate}
\end{lem}

Consider the space $\Omega = G^{\N}$ endowed with  the distribution $\mu$ of the random walk process $(Z_i)_{i \in \N}$. We define the action of $G$ on $\Omega$ by $g \cdot (z_1, z_2, \ldots ) = (g \cdot z_1, g \cdot z_2, \ldots)$. Observe that this action is non-singular, i.e. $g \cdot \mu \ll \mu$ for all $g \in G$. To see this we note that $\nu \ll g \cdot \nu$ (we remind that $\nu$ has full support) and that the pair of measures  $\mu$, $g \cdot \mu$ are obtained from the pair $\nu$, $g \cdot \nu$ by application of the same Markov operator. The important consequence is that if a measurable subset $A \subset \Omega$ has full measure, then for every $g$ and for $\mu$-almost every $\zeta \in \Omega$, we have that both $\zeta$ and  $g \cdot \zeta$ lie in $A$.

\begin{lem}
For almost every trajectory $\zeta = (z_i)$ and for all $g \in G$ there are numbers $n,m > 0$ such that for all $i \geq 0$ we have $(\tau(g \cdot \zeta))_{i+n} = g \cdot (\tau(\zeta))_{i+m}$.
\end{lem}
\begin{proof}
We may assume that $\tau$ is defined for $\zeta$ and $g \cdot \zeta$ and the statement of the previous lemma holds for both, since the action of $G$ is non-singular.
Let $\tau(\zeta) = (s_i)_{i \in \N}$. Observe that $\tau(\zeta) = \bigcup_{i \in \N} t(s_i)$ and for some $N>0$ we have $\tau(g \cdot \zeta) = \bigcup_{i>N} t(g \cdot s_i)$.

Let $q > 0$ be such that $g \in A_q$. Let $m$ be the smallest number such that $\rank(s_m) \geq q$. We observe that for every $i \geq m$

$$
p(g s_{i+1}) = g p(s_{i+1}) = g s_i,
$$
and so
\begin{equation}\label{eq: trec}
t(g s_{i+1}) = t(p(g s_{i+1})) \sqcup p(g s_{i+1}) = t(g s_i) \sqcup g s_i.
\end{equation}
Set $n = \lvert t(g s_{m+1})\rvert$.
Comparing \ref{eq: trec} with 
\begin{equation}\label{eq: trec2}
t(s_{i+1}) = t(s_i) \sqcup s_i,\text{ for }i \geq m,
\end{equation}
we conclude, by induction, that for all $u \geq m$ and for all $i \geq 0$ such that $m+i \leq  u$, we have
\begin{equation}\label{eq: trec2}
(t(g \cdot s_{u+1}))_{n+i} = g \cdot (t(s_{u+1}))_{m+i},
\end{equation}

Now, we remind that there is a growing sequence $(q_i)_{i \in \N}$ such that $s_i = z_{q_i}$, so 

\begin{equation}\label{eq: union1}
\tau(\zeta) = \bigcup_{j > m} t(s_{j+1})
\end{equation}
and 

\begin{equation}\label{eq: union2}
\tau(g \cdot \zeta)  =\bigcup_{j \geq M} t(g \cdot z_j) = \bigcup_{j \geq N} t(g \cdot z_{q_j}) = \bigcup_{j \geq N} t(g \cdot s_j),
\end{equation}
for some $M,N$.

Combining \ref{eq: union1} and \ref{eq: union2} with \ref{eq: trec2}, we get the desired statement.
\end{proof}

\begin{theor}
The action of group $G$ on the tail boundary $\partial(G, \nu)$ is essentially free.
\end{theor}
\begin{proof}
For almost every $\zeta \in \Omega$ and every $g \in G$ there are $m,n$ such that $g \cdot (\tau(\zeta))_{m+i} = (\tau(g \cdot \zeta))_{n+i}$ for all $i \in \N$.
Also, for big enough $i$ we have $(\tau(\zeta))_{m+i} \in W'_j$ for some $j$ such that $g \in A_j$, and hence  $g \cdot (\tau(\zeta))_{m+i}$ is the unique element of the sequence $(\tau(g \cdot \zeta))_{n+i}$ that lies in $W_j$. On the other hand, $(\tau(\zeta))_{m+i}$ is the unique element of the sequence $\tau(\zeta)$ that lies in $W_j$. Hence, $\tau(\zeta) \neq \tau(g \cdot \zeta)$ for a.e. trajectory $\zeta \in \Omega$ and every $g \in G$, $g \neq e$ (since $G$ is countable). We are done since $\tau$ is tail-measurable.
\end{proof}

For any measurable map between measurable spaces $\Lambda \to \Omega$ there is a minimal subalgebra on $\Lambda$ such that the map becomes measurable. We say that this subalgebra is generated by the map.
The following lemma shows that we explicitly characterized the (non-trivial) boundary for our example.

\begin{lem}
The subalgebra generated by the map $\tau$ is isomorphic (up to sets of measure zero) to the subalgebra $\partial(G, \nu)$.
\end{lem}
\begin{proof}
We note that by Lemma \ref{lem: tau properties}, for almost every trajectory $\zeta = (z_1, z_2, \ldots)\in \Omega$ there is growing sequence $(q_i)$ of natural numbers and $N$ such that $s_{i+N} = z_{q_i}$ for all $i \in \N$, where $\tau(\zeta) = (s_1, s_2, \ldots)$. Now the lemma follows from the identification of the boundary with the class of traps, see \cite[Proposition 0.1]{KaVe83}. More precisely, for any subset $E$ of the tail subalgebra there is a subset $\bar{E}$ of $G$ such that $\zeta \in E$ is equivalent to $z_i \in \bar{E}$ for all but  finitely many $i$, for almost all $\zeta$. This means that $E$ is measurable with respect to the subagebra of $\Omega$ generated by function $\tau$.
\end{proof}


\begin{thebibliography}{100000}
\bibitem[DuM56]{DuM56} A.M. Duguid and D.H. McLain,\textit{FC-nilpotent and FC-soluble groups}, Mathematical proceedings of the Cambridge philosophical society, 1956, pp. 391--398.

\bibitem[ErKa19]{ErKa19} A. Erschler and V. Kaimanovich, \textit{Arboreal structures on groups and the associated boundaries}, arXiv preprint arXiv:1903.02095 (2019).

\bibitem[Feta19]{Feta19} J. Frisch,Y Hartman, O. Tamuz, and P. V. Ferdowsi. \textit{Choquet-Deny groups and the infinite conjugacy class property}, Annals of Mathematics 190, no. 1 (2019): 307--320.
\bibitem[Ja]{Ja} W. Jaworski, \textit{Countable amenable identity excluding groups}, Canadian
Mathematical Bulletin 47 (2004), no. 2, 215--228.
\bibitem[Ka83]{Ka83} V.A. Kaimanovich \textit{Examples of non-abelian discrete groups with non-trivial exit boundary}, Zap. Nauchn. Sem. LOMI, 1983, Volume 123, 167--184.
\bibitem[Ka92]{Ka92}V.A.  Kaimanovich, \textit{Measure-theoretic boundaries of Markov chains, 0–2 laws and entropy}, Harmonic analysis and discrete potential theory. Springer, Boston, MA, 1992. 145--180.
\bibitem[KaVe83]{KaVe83} V.A. Kaimanovich, and A. M. Vershik, \textit{Random walks on discrete groups: boundary and entropy}, The annals of probability (1983): 457--490.
\bibitem[M56]{M56} D.H. McLain, \textit{Remarks on the upper central series of a group}, Glasgow Mathematical Journal 3 (1956), no. 1, 38–44.
\bibitem[Pat79]{Pat79} A.L.T. Paterson, \textit{Amenable groups for which every topological left invariant mean is invariant}, Pacific Journal of Mathematics, 1979, pp. 391--397.
\bibitem[Ro81]{Ro81} J. Rosenblatt, \textit{Ergodic and mixing random walks on locally compact groups}, Mathematische Annalen 257 (1981), no. 1, 31--42.
\end{thebibliography}
\end{document}